\documentclass[english]{article}
\usepackage{ae,aecompl}
\usepackage[T1]{fontenc}
\usepackage[latin9]{inputenc}
\usepackage[letterpaper]{geometry}
\usepackage{fancyhdr}
\usepackage{babel}
\usepackage{enumitem}
\usepackage{amsmath}
\usepackage{amsthm}
\usepackage{amssymb}
\usepackage[all]{xy}
\usepackage[unicode=true,
 bookmarks=true,bookmarksnumbered=false,bookmarksopen=false,
 breaklinks=false,pdfborder={0 0 1},backref=false,colorlinks=false]
 {hyperref}

\makeatletter
\theoremstyle{plain}
\newtheorem{thm}{\protect\theoremname}
  \theoremstyle{definition}
  \newtheorem{defn}[thm]{\protect\definitionname}
  \theoremstyle{plain}
  \newtheorem{lem}[thm]{\protect\lemmaname}
  \theoremstyle{plain}
  \newtheorem{cor}[thm]{\protect\corollaryname}
  \theoremstyle{plain}
  \newtheorem{conjecture}[thm]{\protect\conjecturename}

\@ifundefined{date}{}{\date{}}

\AtBeginDocument{
  
}

\makeatother

  \providecommand{\conjecturename}{Conjecture}
  \providecommand{\corollaryname}{Corollary}
  \providecommand{\definitionname}{Definition}
  \providecommand{\lemmaname}{Lemma}
\providecommand{\theoremname}{Theorem}

\begin{document}

\lhead{On the $\textup{PGL}_{2}$-invariant quadruples of torsion points
of elliptic curves}

\rhead{Bogomolov $\cdot$ Fu}

\title{On the $\textup{PGL}_{2}$-invariant quadruples of torsion points
of elliptic curves}

\author{Fedor Bogomolov $\cdot$ Hang Fu}
\maketitle
\begin{quote}
\textbf{\small{}Abstract.}{\small{} Let $E$ be an elliptic curve
and $\pi:E\to\mathbb{P}^{1}$ a standard double cover identifying
$\pm P\in E$. It is known that for some torsion points $P_{i}\in E$,
$1\leq i\leq4$, the cross ratio of $\{\pi(P_{i})\}_{i=1}^{4}$ is
independent of $E$. In this article, we will give a complete classification
of such quadruples.}{\small \par}

\textbf{\small{}Keywords.}{\small{} Elliptic curves $\cdot$ Torsion
points $\cdot$ $q$-series $\cdot$ Congruence subgroups $\cdot$
Modular curves}{\small \par}

\textbf{\small{}Mathematics Subject Classification.}{\small{} 14H52
$\cdot$ 11G05 $\cdot$ 11F03 $\cdot$ 20H05 $\cdot$ 40A20}{\small \par}
\end{quote}

\section{Introduction and the Statement of Main Theorem}

Let $K$ be an algebraically closed field, $E$ an elliptic curve
defined over $K$ with the identity element $O$, $E[n]$ the $n$-torsion
subgroup, $E^{*}[n]\subseteq E[n]$ the collection of torsion points
of order $n$, $E[\infty]=\cup_{n=1}^{\infty}E[n]$ the collection
of all torsion points, and $\pi:E\to\mathbb{P}^{1}(K)$ a standard
double cover identifying $\pm P\in E$.

In this article, we will continue to study the image of $E[\infty]$
under $\pi$. See \cite{MR3536148,MR3799155,Hitchin,MR2349648} for
the background and prior results. In particular, \cite[Conjecture 2]{Hitchin},
the guiding problem of our project, predicts that if $K=\mathbb{C}$
and $\pi_{1}(E_{1}[2])\neq\pi_{2}(E_{2}[2])$, then $\left|\pi_{1}(E_{1}[\infty])\cap\pi_{2}(E_{2}[\infty])\right|$
is bounded by some universal constant. A partial result \cite[Theorem 1.3]{arXiv}
has recently been claimed under the assumption $\left|\pi_{1}(E_{1}[2])\cap\pi_{2}(E_{2}[2])\right|=3$.
On the other hand, in \cite[Theorem 1.1]{MR3799155}, we are able
to find $E_{1}$ and $E_{2}$ such that $\left|\pi_{1}(E_{1}[\infty])\cap\pi_{2}(E_{2}[\infty])\right|\geq22$.
A crucial fact utilized in our construction is that, up to an automorphism
of $\mathbb{P}^{1}(K)$, $\pi(E^{*}[4])$ is independent of $E$.
The same statement is also correct if $\pi(E^{*}[4])$ is replaced
with $\pi(E^{*}[3])$. These two interesting examples raise the question
of when such phenomena may happen. In this article, we will establish
a modified version of \cite[Conjecture 3]{Hitchin} as Theorem \ref{thm1}.
We note that the original statement of \cite[Conjecture 3]{Hitchin}
is inaccurate due to the infinitude of cases (\ref{case1}) and (\ref{case2}).

Let
\[
\mathcal{R}=(\{0\}\times[0,1/2]\cup(0,1/2)\times[0,1)\cup\{1/2\}\times[0,1/2])\cap\mathbb{Q}^{2}
\]
be a fundamental domain of $\pm\backslash\mathbb{Q}^{2}/\mathbb{Z}^{2}$.
There is a natural left $\textup{SL}_{2}(\mathbb{Z})$-action on the
fourth unordered configuration space of $\mathcal{R}$,
\[
\mathcal{R}_{4}=\{\{(r_{i},\theta_{i})\}_{i=1}^{4}:(r_{i},\theta_{i})\in\mathcal{R},(r_{i},\theta_{i})\neq(r_{j},\theta_{j})\textup{ for }i\neq j\}.
\]
For $\gamma\in\textup{SL}_{2}(\mathbb{Z})$ and $\{(r_{i},\theta_{i})\}_{i=1}^{4}\in\mathcal{R}_{4}$,
we define $\gamma\cdot\{(r_{i},\theta_{i})\}_{i=1}^{4}=\{F((r_{i},\theta_{i})\cdot\gamma^{T})\}_{i=1}^{4}$,
where $F(r,\theta)\in\mathcal{R}$ is the unique element such that
$F(r,\theta)\equiv\pm(r,\theta)\textup{ mod }\mathbb{Z}^{2}$.
\begin{thm}
\label{thm1}Let $S=\{(r_{i},\theta_{i})\}_{i=1}^{4}\in\mathcal{R}_{4}$,
$n$ their common order, $\{w_{1},w_{2}\}$ a basis of $E[n]$. Write
$E[r_{i},\theta_{i}]=[nr_{i}]w_{1}+[n\theta_{i}]w_{2}$. Then the
image of $\{E[r_{i},\theta_{i}]\}_{i=1}^{4}$ under $\pi$ inside
$\mathfrak{S}_{4}\backslash(\mathbb{P}^{1})^{4}/\textup{PGL}_{2}$
is a constant (independent of $K$, $E$, $\pi$, and $\{w_{1},w_{2}\}$)
if and only if $S$ is $\textup{SL}_{2}(\mathbb{Z})$-equivalent to
exactly one of the following:
\begin{enumerate}[label=\textup{(\arabic*)},ref=\arabic*]
\item \label{case1} $\{(0,a),(0,1/4),(0,1/2-a),(1/2,1/4)\}$, where $a\in\{0\}\cup\{1/(2r):r\in\mathbb{Z},r\geq3\}$,
\item \label{case2} $\{(0,b),(1/4,0),(1/4,1/2),(1/2,b)\}$, where $b\in\{1/(2r):r\in\mathbb{Z},r\geq2\}$,
\item \label{case3} $\{(0,1/3),(1/3,0),(1/3,1/3),(1/3,2/3)\}$,
\item \label{case4}  $\{(0,1/4),(1/4,0),(1/4,1/4),(1/4,1/2)\}$,
\item \label{case5} $\{(0,1/4),(1/2,0),(1/2,1/4),(1/2,1/2)\}$,
\item \label{case6} $\{(0,1/6),(1/6,0),(1/6,1/6),(1/3,2/3)\}$,
\item \label{case7} $\{(0,1/6),(1/6,0),(1/6,1/3),(1/3,5/6)\}$,
\item \label{case8} $\{(0,1/6),(1/3,0),(1/3,1/6),(1/3,1/3)\}$,
\item \label{case9} $\{(0,1/6),(1/3,1/6),(1/3,1/2),(1/3,5/6)\}$,
\item \label{case10} $\{(0,1/8),(1/4,1/8),(1/4,3/8),(1/2,1/8)\}$,
\item \label{case11} $\{(0,1/12),(1/3,1/12),(1/3,11/12),(1/2,1/4)\}$.
\end{enumerate}
Moreover, if we fix the quotient map
\begin{eqnarray*}
(\mathbb{P}^{1})^{4}\;\xrightarrow{\textup{cross ratio}}\;(\mathbb{Z}/2\mathbb{Z})^{2}\backslash(\mathbb{P}^{1})^{4}/\textup{PGL}_{2} & \xrightarrow{\mathfrak{S}_{3}} & \mathfrak{S}_{4}\backslash(\mathbb{P}^{1})^{4}/\textup{PGL}_{2}\\
z & \mapsto & \frac{(z^{2}-z+1)^{3}}{z^{2}(z-1)^{2}},
\end{eqnarray*}
then the constant is:
\begin{itemize}
\item $27/4$ for the cases (\ref{case1}), (\ref{case2}), and (\ref{case5}),
\item $0$ for the cases (\ref{case3}), (\ref{case6}), (\ref{case9}),
and (\ref{case11}),
\item $1/2$ for the cases (\ref{case4}) and (\ref{case10}),
\item $8/3$ for the cases (\ref{case7}) and (\ref{case8}).
\end{itemize}
\end{thm}

We will prove Theorem \ref{thm1} in three steps: (I) we obtain a
necessary condition for $S$ to give a constant image; (II) we find
all $S$ that satisfy the necessary condition found in (I); and (III)
we prove that all $S$ found in (II) indeed give a constant image.

\section{Proof of Theorem \texorpdfstring{\ref{thm1}}{1}: Step (I)}

Since we simply attempt to find a necessary condition in this section,
we can temporarily assume that $K=\mathbb{C}$, so that the analytic
uniformization can be applied.
\begin{thm}
\textup{\label{thm2}\cite[Page 410, Theorem 1.1]{MR1312368}} For
$u,q\in\mathbb{C}$ with $\left|q\right|<1$, define quantities
\begin{eqnarray*}
s_{k}(q) & = & \sum_{n\geq1}\frac{n^{k}q^{n}}{1-q^{n}},\\
a_{4}(q) & = & -5s_{3}(q),\\
a_{6}(q) & = & -\frac{5s_{3}(q)+7s_{5}(q)}{12},\\
X(u,q) & = & \sum_{n\in\mathbb{Z}}\frac{q^{n}u}{(1-q^{n}u)^{2}}-2s_{1}(q),\\
Y(u,q) & = & \sum_{n\in\mathbb{Z}}\frac{(q^{n}u)^{2}}{(1-q^{n}u)^{3}}+s_{1}(q),\\
E_{q} & : & y^{2}+xy=x^{3}+a_{4}(q)x+a_{6}(q).
\end{eqnarray*}
\begin{itemize}
\item $E_{q}$ is an elliptic curve, and $X$ and $Y$ define a complex
analytic isomorphism
\begin{eqnarray*}
\phi:\mathbb{C}^{*}/q^{\mathbb{Z}} & \to & E_{q}(\mathbb{C})\\
u & \mapsto & \begin{cases}
(X(u,q),Y(u,q)) & \textup{if }u\notin q^{\mathbb{Z}},\\
O & \textup{if }u\in q^{\mathbb{Z}}.
\end{cases}
\end{eqnarray*}
\item For every elliptic curve $E/\mathbb{C}$ there is a $q\in\mathbb{C}^{*}$
with $\left|q\right|<1$ such that $E$ is isomorphic to $E_{q}$.
\end{itemize}
\end{thm}

\begin{defn}
\label{def3}$ $
\begin{itemize}
\item For any $r\in\mathbb{Q}$, define $F(r)\in[0,1/2]$ to be the unique
element such that $F(r)\equiv\pm r\textup{ mod }\mathbb{Z}$. A set
$\{r_{1},r_{2},r_{3},r_{4}\}\subseteq\mathbb{Q}$ is said to be ``good''
if
\[
F(r_{\sigma(1)})\leq F(r_{\sigma(2)})<F(r_{\sigma(3)})\leq F(r_{\sigma(4)})
\]
for some permutation $\sigma\in\mathfrak{S}_{4}$.
\item For any $(r,\theta)\in\mathbb{Q}^{2}$, define $F(r,\theta)\in\mathcal{R}$
to be the unique element such that $F(r,\theta)\equiv\pm(r,\theta)\textup{ mod }\mathbb{Z}^{2}$.
A set $\{(r_{i},\theta_{i})\}_{i=1}^{4}\subseteq\mathbb{Q}^{2}$ is
said to be ``good'' if there exist $a,b\in\mathbb{Z}$ such that
$\gcd(a,b)=1$ and $\{ar_{i}+b\theta_{i}\}_{i=1}^{4}$ is ``good''.
\end{itemize}
\end{defn}

\begin{lem}
\label{lem4}Let $u_{i}=\exp(2\pi\sqrt{-1}\theta_{i})q^{r_{i}}$,
$1\leq i\leq4$, be four distinct torsion points on $\pm\backslash\mathbb{C}^{*}/q^{\mathbb{Z}}$.
If $\{r_{1},r_{2},r_{3},r_{4}\}$ is ``good'', then the cross ratio
of $\{X(u_{i},q)\}_{i=1}^{4}$ is not a constant function of $q$.
\end{lem}

\begin{proof}
Assume that $(r_{i},\theta_{i})\in\mathcal{R}$ and they are ordered
lexicographically. First note that
\[
s_{1}(q)=\sum_{n\geq1}\frac{nq^{n}}{1-q^{n}}=\sum_{n\geq1}n\sum_{m\geq1}(q^{n})^{m}=\sum_{n\geq1}\sum_{m\geq1}m(q^{n})^{m},
\]
then
\begin{eqnarray*}
X(u,q) & = & \sum_{n\in\mathbb{Z}}\frac{q^{n}u}{(1-q^{n}u)^{2}}-2s_{1}(q)\\
 & = & \frac{u}{(1-u)^{2}}+\sum_{n\geq1}\left(\frac{q^{n}u}{(1-q^{n}u)^{2}}+\frac{q^{n}u^{-1}}{(1-q^{n}u^{-1})^{2}}\right)-2s_{1}(q)\\
 & = & \frac{u}{(1-u)^{2}}+\sum_{n\geq1}\sum_{m\geq1}\left(m(q^{n}u)^{m}+m(q^{n}u^{-1})^{m}-2m(q^{n})^{m}\right)\\
 & = & \begin{cases}
(\zeta+\zeta^{-1}-2)^{-1}+\textup{higher degree terms of }q & \textup{if }u=\zeta,\\
\zeta q^{r}+\textup{higher degree terms of }q & \textup{if }u=\zeta q^{r}\textup{ with }0<r<1/2,\\
(\zeta+\zeta^{-1})q^{1/2}+\textup{higher degree terms of }q & \textup{if }u=\zeta q^{1/2}\textup{ with }\zeta\neq\sqrt{-1},\\
-6q+\textup{higher degree terms of }q & \textup{if }u=\sqrt{-1}q^{1/2}.
\end{cases}
\end{eqnarray*}
In particular, we find that if $u_{i}\ne u_{j}$, then the lowest
degree terms of $X(u_{i},q)$ and $X(u_{j},q)$ are different. If
$u_{1}\neq1$, then the cross ratio of $\{X(u_{i},q)\}_{i=1}^{4}$
is
\[
\frac{(X(u_{1},q)-X(u_{2},q))(X(u_{3},q)-X(u_{4},q))}{(X(u_{1},q)-X(u_{4},q))(X(u_{3},q)-X(u_{2},q))}=cq^{r_{3}-r_{2}}+\textup{higher degree terms of }q\textup{ for some }c\neq0.
\]
If $u_{1}=1$, then the cross ratio of $\{X(u_{i},q)\}_{i=1}^{4}$
is
\[
\frac{X(u_{3},q)-X(u_{4},q)}{X(u_{3},q)-X(u_{2},q)}=cq^{r_{3}-r_{2}}+\textup{higher degree terms of }q\textup{ for some }c\neq0.
\]
Since $r_{2}<r_{3}$ by the assumption, the cross ratio is nonconstant.
\end{proof}
\begin{cor}
\label{cor5}Let $u_{i}=\exp(2\pi\sqrt{-1}\theta_{i})q^{r_{i}}$,
$1\leq i\leq4$, be four distinct torsion points on $\pm\backslash\mathbb{C}^{*}/q^{\mathbb{Z}}$.
If $\{(r_{i},\theta_{i})\}_{i=1}^{4}$ is ``good'', then the cross
ratio of $\{X(u_{i},q)\}_{i=1}^{4}$ is not a constant function of
$q$.
\end{cor}

\begin{proof}
By Definition \ref{def3}, there exist $a,b\in\mathbb{Z}$ such that
$\gcd(a,b)=1$ and $\{ar_{i}+b\theta_{i}\}_{i=1}^{4}$ is ``good''.
Since $\gcd(a,b)=1$, there exist $c,d\in\mathbb{Z}$ such that $\gamma=\begin{pmatrix}a & b\\
c & d
\end{pmatrix}\in\textup{SL}_{2}(\mathbb{Z})$. Let $(r'_{i},\theta'_{i})=(r{}_{i},\theta{}_{i})\cdot\gamma^{T}$,
$\tau=\dfrac{\log(q)}{2\pi\sqrt{-1}}$, $\tau'=\gamma^{-T}\cdot\tau=\dfrac{d\tau-c}{-b\tau+a}$,
$q'=\exp\left(2\pi\sqrt{-1}\tau'\right)$, and $u'_{i}=\exp(2\pi\sqrt{-1}\theta'_{i})q'^{r'_{i}}$.
Then by \cite[Page 50, Theorem 6.2]{MR1312368},
\begin{eqnarray*}
X(u_{i},q) & = & -\frac{\wp(r_{i}\tau+\theta_{i},\mathbb{Z}\tau+\mathbb{Z})}{4\pi^{2}}-\frac{1}{12}\\
 & = & -\frac{\wp(r'_{i}(d\tau-c)+\theta'_{i}(-b\tau+a),\mathbb{Z}(d\tau-c)+\mathbb{Z}(-b\tau+a))}{4\pi^{2}}-\frac{1}{12}\\
 & = & -\frac{\wp(r'_{i}\tau'+\theta'_{i},\mathbb{Z}\tau'+\mathbb{Z})}{4\pi^{2}(-b\tau+a)^{2}}-\frac{1}{12}\\
 & = & -\frac{-4\pi^{2}(X(u'_{i},q')+\frac{1}{12})}{4\pi^{2}(-b\tau+a)^{2}}-\frac{1}{12}\\
 & = & \frac{X(u'_{i},q')+\frac{1}{12}}{(-b\tau+a)^{2}}-\frac{1}{12}.
\end{eqnarray*}
Since $\{r'_{1},r'_{2},r'_{3},r'_{4}\}$ is ``good'' by the assumption,
the cross ratio of $\{X(u'_{i},q')\}_{i=1}^{4}$ is nonconstant, and
the cross ratio of $\{X(u_{i},q)\}_{i=1}^{4}$ is also nonconstant.
\end{proof}

\section{Proof of Theorem \texorpdfstring{\ref{thm1}}{1}: Step (II)}
\begin{lem}
\label{lem6}$S\in\mathcal{R}_{4}$ is not ``good'' if and only
if $S$ is $\textup{SL}_{2}(\mathbb{Z})$-equivalent to exactly one
of the quadruples listed in Theorem \ref{thm1}.
\end{lem}

\begin{proof}
Suppose that $S$ is not ``good''. Without loss of generality, we
assume that $S$ is the minimal element within its $\textup{SL}_{2}(\mathbb{Z})$-orbit
with respect to the lexicographical order. It must be of the form
\[
S=\{(0,\theta_{1}),(r,\theta_{2}),(r,\theta_{3}),(r_{4},\theta_{4})\}\in\mathcal{R}_{4},
\]
where $0\leq r\leq r_{4}\leq1/2$ and $\theta_{2}\neq\theta_{3}$.
By the minimality of $S$, if $\theta_{1}=0$, then $r=0$; if $\theta_{1}\neq0$,
then $1/\theta_{1}=\textup{ord}(\theta_{1})\geq\textup{ord}((r,\theta_{2})),\textup{ord}((r,\theta_{3})),\textup{ord}((r_{4},\theta_{4}))$.
Let the lower triangular matrix $\begin{pmatrix}1 & 0\\
c & 1
\end{pmatrix}\in\textup{SL}_{2}(\mathbb{Z})$ act on $S$, we get
\[
\{(0,\theta_{1}),(r,rc+\theta_{2}),(r,rc+\theta_{3}),(r_{4},r_{4}c+\theta_{4})\}.
\]
Consider the following eight sets
\[
\begin{aligned} & A_{1}=\{c\in\mathbb{Z}:rc+\theta_{2}\equiv0\textup{ mod }1\}, &  & d_{1}=\textup{ord}(r),\\
 & A_{2}=\{c\in\mathbb{Z}:rc+\theta_{3}\equiv0\textup{ mod }1\}, &  & d_{2}=\textup{ord}(r),\\
 & A_{3}=\{c\in\mathbb{Z}:r_{4}c+\theta_{4}\equiv0\textup{ mod }1\}, &  & d_{3}=\textup{ord}(r_{4}),\\
 & A_{4}=\{c\in\mathbb{Z}:2rc+\theta_{2}+\theta_{3}\equiv0\textup{ mod }1\}, &  & d_{4}=\textup{ord}(2r),\\
 & A_{5}=\{c\in\mathbb{Z}:(r_{4}+r)c+\theta_{4}+\theta_{2}\equiv0\textup{ mod }1\}, &  & d_{5}=\textup{ord}(r_{4}+r),\\
 & A_{6}=\{c\in\mathbb{Z}:(r_{4}+r)c+\theta_{4}+\theta_{3}\equiv0\textup{ mod }1\}, &  & d_{6}=\textup{ord}(r_{4}+r),\\
 & A_{7}=\{c\in\mathbb{Z}:(r_{4}-r)c+\theta_{4}-\theta_{2}\equiv0\textup{ mod }1\}, &  & d_{7}=\textup{ord}(r_{4}-r),\\
 & A_{8}=\{c\in\mathbb{Z}:(r_{4}-r)c+\theta_{4}-\theta_{3}\equiv0\textup{ mod }1\}, &  & d_{8}=\textup{ord}(r_{4}-r).
\end{aligned}
\]
For any $1\leq i\leq8$, $A_{i}$ is an empty set or an arithmetic
progression. Whenever $A_{i}\neq\varnothing$, $d_{i}$ is the difference
of its consecutive terms. We note that if $d_{1}$ is even, then $d_{1}=2d_{4}$;
if $d_{1}$ is odd, then $d_{1}=d_{4}$.

We claim that $\cup_{i=1}^{8}A_{i}=\mathbb{Z}$ and $\sum_{i=1}^{8}1/d_{i}\geq1$.\\
Take any $c\in\mathbb{Z}$, we know that $\{\theta_{1},rc+\theta_{2},rc+\theta_{3},r_{4}c+\theta_{4}\}$
is not ``good''. If $c\notin\cup_{i=1}^{3}A_{i}$, then by the minimality
of $S$, $\min\{F(\theta_{1}),F(rc+\theta_{2}),F(rc+\theta_{3}),F(r_{4}c+\theta_{4})\}=F(\theta_{1})$.
Therefore, two of $F(rc+\theta_{2})$, $F(rc+\theta_{3})$, and $F(r_{4}c+\theta_{4})$
must be equal, which exactly means $c\in\cup_{i=4}^{8}A_{i}$. If
$A_{i}\neq\varnothing$, then $A_{i}$ covers $1/d_{i}$ of $\mathbb{Z}$.
Since $\cup_{i=1}^{8}A_{i}=\mathbb{Z}$, we have $\sum_{i=1}^{8}1/d_{i}\geq1$.

Now we use this criterion to find all possible $S$ with the help
of \textbf{computer programs} \cite{Mathematica}.
\begin{itemize}
\item Assume that $\min\{d_{i}\}_{i=1}^{8}\geq4$.\\
Let $B_{1}=\{(r,r_{4})\in\mathbb{Q}^{2}:0\leq r\leq r_{4}\leq1/2,\sum_{i=1}^{8}1/d_{i}\geq1,\min\{d_{i}\}_{i=1}^{8}\geq4\}$.
\begin{itemize}
\item If $d_{5},d_{7}\geq27$, then $\sum_{i=1}^{8}1/d_{i}\leq1/4+3/5+4/27=539/540<1$.
\item Let $B_{11}=\{(r,r_{4})\in B_{1}:d_{5},d_{7}\leq26\}$.
\item If $4\leq d_{5}\leq26$, $d_{7}\geq27$, and $d_{1}\geq23$, then
$\sum_{i=1}^{8}1/d_{i}\leq3/4+4/23+2/27=2479/2484<1$.
\item Let $B_{12}=\{(r,r_{4})\in B_{1}:d_{1}\leq22,d_{5}\leq26\}$.
\item If $4\leq d_{7}\leq26$, $d_{5}\geq27$, and $d_{1}\geq23$, then
$\sum_{i=1}^{8}1/d_{i}\leq3/4+4/23+2/27=2479/2484<1$.
\item Let $B_{13}=\{(r,r_{4})\in B_{1}:d_{1}\leq22,d_{7}\leq26\}$.
\end{itemize}
\item Assume that $\min\{d_{i}\}_{i=1}^{8}=3$ and $A_{1}\cup A_{2}\cup A_{4}\neq\mathbb{Z}$.\\
Let $B_{2}=\{(r,r_{4})\in\mathbb{Q}^{2}:0\leq r\leq r_{4}\leq1/2,\sum_{i=1}^{8}1/d_{i}\geq1,\min\{d_{i}\}_{i=1}^{8}\geq3\}$.
\begin{itemize}
\item If $d_{5}=3$ and $d_{7}=n$, then $d_{1},d_{3},d_{4}\geq n/3$. If
$n>42$, then $\sum_{i=1}^{8}1/d_{i}\leq2/3+14/n<1$.
\item Let $B_{21}=\{(r,r_{4})\in B_{2}:d_{5}=3,d_{7}\leq42\}$.
\item If $d_{7}=3$ and $d_{5}=n$, then $d_{1},d_{3},d_{4}\geq n/3$. If
$n>42$, then $\sum_{i=1}^{8}1/d_{i}\leq2/3+14/n<1$.
\item Let $B_{22}=\{(r,r_{4})\in B_{2}:d_{7}=3,d_{5}\leq42\}$.
\item If $r_{4}=1/3$ and $d_{1}=n$, then $d_{5},d_{7}\geq n/3$. If $n>24$,
then $\sum_{i=1}^{8}1/d_{i}\leq1/3+16/n<1$.
\item Let $B_{23}=\{(r,r_{4})\in B_{2}:r_{4}=1/3,d_{1}\le24\}$.
\item If $r=1/6$ and $d_{3}=n$, then $d_{5},d_{7}\geq n/6$. If $n>75$,
then $\sum_{i=1}^{8}1/d_{i}\leq2/3+25/n<1$.
\item Let $B_{24}=\{(r,r_{4})\in B_{2}:r=1/6,d_{3}\le75\}$.
\item If $r=1/3$ and $d_{3}=n$, then $d_{5},d_{7}\geq n/3$. Since we
assumed $A_{1}\cup A_{2}\cup A_{4}\neq\mathbb{Z}$, we only need $\sum_{i\neq1}1/d_{i}\geq1$.
If $n>39$, then $\sum_{i\neq1}1/d_{i}\leq2/3+13/n<1$.
\item Let $B_{25}=\{(r,r_{4})\in B_{2}:r=1/3,d_{3}\le39,\sum_{i\neq1}1/d_{i}\geq1\}$.
\end{itemize}
\item Assume that $\min\{d_{i}\}_{i=1}^{8}=2$, $A_{1}\cup A_{2}\cup A_{4}\neq\mathbb{Z}$,
and $A_{5}\cup A_{6}\neq\mathbb{Z}$.\\
Let $B_{3}=\{(r,r_{4})\in\mathbb{Q}^{2}:0\leq r\leq r_{4}\leq1/2,\sum_{i=1}^{8}1/d_{i}\geq1,\min\{d_{i}\}_{i=1}^{8}\geq2\}$.
\begin{itemize}
\item If $d_{5}=2$ and $d_{7}=n$, then $d_{1},d_{3},d_{4}\geq n/2$. Since
we assumed $A_{5}\cup A_{6}\neq\mathbb{Z}$, we only need $\sum_{i\neq5}1/d_{i}\geq1$.
If $n>20$, then $\sum_{i\neq5}1/d_{i}\leq1/2+10/n<1$.
\item Let $B_{31}=\{(r,r_{4})\in B_{3}:d_{5}=2,d_{7}\leq20,\sum_{i\neq5}1/d_{i}\geq1\}$.
\item If $r_{4}=1/2$ and $d_{1}=n$, then $d_{5},d_{7}\geq n/2$. If $n>24$,
then $\sum_{i=1}^{8}1/d_{i}\leq1/2+12/n<1$.
\item Let $B_{32}=\{(r,r_{4})\in B_{3}:r_{4}=1/2,d_{1}\le24\}$.
\item If $r=1/4$ and $d_{3}=n$, then $d_{5},d_{7}\geq n/4$. Since we
assumed $A_{1}\cup A_{2}\cup A_{4}\neq\mathbb{Z}$, we only need $\sum_{i\neq1}1/d_{i}\geq1$.
If $n>68$, then $\sum_{i\neq1}1/d_{i}\leq3/4+17/n<1$.
\item Let $B_{33}=\{(r,r_{4})\in B_{3}:r=1/4,d_{3}\le68,\sum_{i\neq1}1/d_{i}\geq1\}$.
\end{itemize}
\item Assume that $\min\{d_{i}\}_{i=1}^{8}=1$ and $r\neq0$.
\begin{itemize}
\item Since we assumed $r\neq0$, we have $r=r_{4}$ and $A_{7},A_{8}=\varnothing$.
If $d_{1}$ is even, we need $\sum_{i=1}^{6}1/d_{i}=9/d_{1}\geq1$;
if $d_{1}$ is odd, we need $\sum_{i=1}^{6}1/d_{i}=6/d_{1}\geq1$.
\item Let $B_{41}=\{(r,r_{4})\in\mathbb{Q}^{2}:0\leq r=r_{4}\leq1/2,d_{1}=2,3,4,5,6,8\}$.
\end{itemize}
\end{itemize}
Let $B=B_{11}\cup B_{12}\cup B_{13}\cup B_{21}\cup B_{22}\cup B_{23}\cup B_{24}\cup B_{25}\cup B_{31}\cup B_{32}\cup B_{33}\cup B_{41}$,
which is a finite set. If $S$ is not ``good'', then we have several
possibilities: $(r,r_{4})\in B$, $A_{1}\cup A_{2}\cup A_{4}=\mathbb{Z}$,
$A_{5}\cup A_{6}=\mathbb{Z}$, or $r=0$. In the following, whenever
we consider $(r,r_{4})\in B$, we assume that $A_{1}\cup A_{2}\cup A_{4}\neq\mathbb{Z}$
and $A_{5}\cup A_{6}\neq\mathbb{Z}$.

We claim that
\begin{itemize}
\item if $\gcd(d_{1},d_{3})=1$, then ($A_{1}=\varnothing$ or $A_{3}=\varnothing$)
and ($A_{2}=\varnothing$ or $A_{3}=\varnothing$);
\item if $\gcd(d_{1},d_{5})=1$, then ($A_{1}=\varnothing$ or $A_{5}=\varnothing$)
and ($A_{2}=\varnothing$ or $A_{6}=\varnothing$);
\item if $\gcd(d_{1},d_{7})=1$, then ($A_{1}=\varnothing$ or $A_{7}=\varnothing$)
and ($A_{2}=\varnothing$ or $A_{8}=\varnothing$);
\item if $\gcd(d_{3},d_{5})=1$, then ($A_{3}=\varnothing$ or $A_{5}=\varnothing$)
and ($A_{3}=\varnothing$ or $A_{6}=\varnothing$);
\item if $\gcd(d_{3},d_{7})=1$, then ($A_{3}=\varnothing$ or $A_{7}=\varnothing$)
and ($A_{3}=\varnothing$ or $A_{8}=\varnothing$);
\item if $\gcd(d_{5},d_{7})=1$, then ($A_{5}=\varnothing$ or $A_{7}=\varnothing$)
and ($A_{6}=\varnothing$ or $A_{8}=\varnothing$).
\end{itemize}
Let us prove the first one, others can be proved similarly. If $\gcd(d_{1},d_{3})=1$
and $A_{1},A_{3}\neq\varnothing$, then $A_{1}\cap A_{3}\neq\varnothing$,
i.e., there exists $c\in\mathbb{Z}$ such that $F(rc+\theta_{2})=F(r_{4}c+\theta_{4})=0$.
Since $\{\theta_{1},rc+\theta_{2},rc+\theta_{3},r_{4}c+\theta_{4}\}$
is not ``good'', if $\theta_{1}\neq0$, then $F(rc+\theta_{3})=0$,
which implies $\theta_{2}=\theta_{3}$, contradiction; if $\theta_{1}=0$,
then $r=0$, which implies $(0,\theta_{1})=(r,\theta_{2})=(0,0)$,
contradiction.

Based on this observation, we define
\begin{itemize}
\item $C_{1}=\{(r,r_{4})\in B:\gcd(d_{1},d_{3})=1,\sum_{i\neq3}1/d_{i}<1,\sum_{i\neq1,2}1/d_{i}<1\}$;
\item $C_{2}=\{(r,r_{4})\in B:\gcd(d_{1},d_{5})=1,\sum_{i\neq1,2}1/d_{i}<1,\sum_{i\neq5,6}1/d_{i}<1,\sum_{i\neq1,5}1/d_{i}<1\}$;
\item $C_{3}=\{(r,r_{4})\in B:\gcd(d_{1},d_{7})=1,\sum_{i\neq1,2}1/d_{i}<1,\sum_{i\neq7,8}1/d_{i}<1,\sum_{i\neq1,7}1/d_{i}<1\}$;
\item $C_{4}=\{(r,r_{4})\in B:\gcd(d_{3},d_{5})=1,\sum_{i\neq3}1/d_{i}<1,\sum_{i\neq5,6}1/d_{i}<1\}$;
\item $C_{5}=\{(r,r_{4})\in B:\gcd(d_{3},d_{7})=1,\sum_{i\neq3}1/d_{i}<1,\sum_{i\neq7,8}1/d_{i}<1\}$;
\item $C_{6}=\{(r,r_{4})\in B:\gcd(d_{5},d_{7})=1,\sum_{i\neq5,6}1/d_{i}<1,\sum_{i\neq7,8}1/d_{i}<1,\sum_{i\neq5,7}1/d_{i}<1\}$.
\end{itemize}
Let $C=B\backslash(C_{1}\cup C_{2}\cup C_{3}\cup C_{4}\cup C_{5}\cup C_{6})$.
If $S$ is not ``good'', then we have several possibilities: $(r,r_{4})\in C$,
$A_{1}\cup A_{2}\cup A_{4}=\mathbb{Z}$, $A_{5}\cup A_{6}=\mathbb{Z}$,
or $r=0$.

Assume that $(r,r_{4})\in C$ and $r\neq r_{4}$. If $\cup_{i=4}^{8}A_{i}\neq\mathbb{Z}$,
then $\cup_{i=1}^{3}A_{i}\neq\varnothing$. Since we assumed $\min\{d_{i}\}_{i=1}^{8}\geq2$
and $A_{5}\cup A_{6}\neq\mathbb{Z}$, any two of $\{A_{i}\}_{i=4}^{8}$
cannot cover $\mathbb{Z}$. If at least three of $\{A_{i}\}_{i=4}^{8}$
cover $\mathbb{Z}$, then $\theta_{2},\theta_{3},\theta_{4}\in(1/(2n))\mathbb{Z}/\mathbb{Z}$,
where $n=\textup{lcm}(d_{4},d_{5},d_{7})$. If $d_{1}=2n$ or $d_{3}=2n$,
then $\cup_{i=1}^{3}A_{i}\neq\varnothing$.

Assume that $(r,r_{4})\in C$ and $r=r_{4}$. If $d_{1}=5,8$, then
$\cup_{i=4}^{6}A_{i}\neq\mathbb{Z}$, so $\cup_{i=1}^{3}A_{i}\neq\varnothing$.
If $d_{1}=6$, then any two of $\{A_{i}\}_{i=4}^{6}$ cannot cover
$\mathbb{Z}$ and $d_{1}=2d_{4}$, so $\cup_{i=1}^{3}A_{i}\neq\varnothing$.
\begin{itemize}
\item Let $D_{1}=\{(r,r_{4})\in C:r\neq r_{4},\sum_{i=4}^{8}1/d_{i}<1$
or $(d_{5}=2,\sum_{i=4,6,7,8}1/d_{i}<1)$ or $d_{1}=2\textup{lcm}(d_{4},d_{5},d_{7})$
or $d_{3}=2\textup{lcm}(d_{4},d_{5},d_{7})\}\cup\{(r,r_{4})\in C:r=r_{4},d_{1}=5,6,8\}$.
\end{itemize}
Assume that $(r,r_{4})\in C$ and $\cup_{i=1}^{3}A_{i}\neq\varnothing$.
If
\[
\begin{aligned} & A_{1}\cup A_{2}\cup A_{4}\neq\mathbb{Z}, &  & A_{1}\cup A_{3}\cup A_{5}\cup A_{7}\neq\mathbb{Z}, &  & A_{2}\cup A_{3}\cup A_{6}\cup A_{8}\neq\mathbb{Z},\\
 & A_{3}\cup A_{4}\neq\mathbb{Z}, &  & A_{2}\cup A_{5}\cup A_{7}\neq\mathbb{Z}, &  & A_{1}\cup A_{6}\cup A_{8}\neq\mathbb{Z},
\end{aligned}
\]
then $\theta_{2},\theta_{3},\theta_{4}\in(1/n)\mathbb{Z}/\mathbb{Z}$,
where $n=\textup{lcm}(d_{1},d_{3})$. Since $\{\theta_{1},rc+\theta_{2},rc+\theta_{3},r_{4}c+\theta_{4}\}$
is not ``good'', if, for example, $F(rc+\theta_{2})=0$, then by
the minimality of $S$, $F(\theta_{1})=F(rc+\theta_{3})$ or $F(\theta_{1})=F(r_{4}c+\theta_{4})$,
which implies $\theta_{1}\in(1/n)\mathbb{Z}/\mathbb{Z}$ as well.
\begin{itemize}
\item Let $D_{2}=\{(r,r_{4})\in C:\sum_{i=1,3,5,7}1/d_{i}<1\textup{ or }(r=r_{4},\sum_{i=1,3,5}1/d_{i}<1),\sum_{i=3,4}1/d_{i}<1\}$.
\end{itemize}
We want to check that for any $(r,r_{4})\in D_{1}\cap D_{2}$, any
$\theta_{1},\theta_{2},\theta_{3},\theta_{4}\in(1/n)\mathbb{Z}/\mathbb{Z}$,
and any $a,b\in\mathbb{Z}$ such that $\gcd(a,b,n)=1$, where $n=\textup{lcm}(d_{1},d_{3})$,
whether $\{b\theta_{1},ar+b\theta_{2},ar+b\theta_{3},ar_{4}+b\theta_{4}\}$
is ``good'' or not. This is a finite calculation, so we can use
\textbf{computer programs} to do that. The answer is: if $S$ is not
``good'', then $(r,r_{4})\notin D_{1}\cap D_{2}$.

By the \textbf{computer programs},
\begin{align*}
 & D_{2}\backslash D_{1}=\{(1/9,2/9),(1/9,4/9),(1/6,5/12),(1/5,2/5),(1/5,7/15),(2/9,4/9),(1/4,3/8),(1/3,5/12)\},\\
 & D_{1}\backslash D_{2}=\{(1/12,5/12),(1/10,1/2),(1/8,3/8),(1/6,1/2),(3/10,1/2)\},\\
 & C\backslash(D_{1}\cup D_{2})=\{(1/6,1/3),(1/4,1/4),(1/4,1/2),(1/3,1/3),(1/3,1/2),(1/2,1/2)\}.
\end{align*}
If $S$ is not ``good'', then we have several possibilities: $(r,r_{4})\in C\backslash(D_{1}\cap D_{2})$,
$A_{1}\cup A_{2}\cup A_{4}=\mathbb{Z}$, $A_{5}\cup A_{6}=\mathbb{Z}$,
or $r=0$.

Assume that $(r,r_{4})\notin D_{1}$, $r\neq r_{4}$, and $\cup_{i=1}^{3}A_{i}=\varnothing$.
Then $\cup_{i=4}^{8}A_{i}=\mathbb{Z}$ and $\theta_{2},\theta_{3},\theta_{4}\in(1/(2n))\mathbb{Z}/\mathbb{Z}$,
where $n=\textup{lcm}(d_{4},d_{5},d_{7})$. Also, for any $c\in\mathbb{Z}$,
$\min\{F(\theta_{1}),F(rc+\theta_{2}),F(rc+\theta_{3}),F(r_{4}c+\theta_{4})\}=F(\theta_{1})$,
so the smallest two of $F(rc+\theta_{2})$, $F(rc+\theta_{3})$, and
$F(r_{4}c+\theta_{4})$ must be equal. We can use \textbf{computer
programs} to find all $(\theta_{2},\theta_{3},\theta_{4})$ that satisfy
these conditions. The answer is:
\begin{itemize}
\item If $(r,r_{4})=(1/5,2/5)$, then $(\theta_{2},\theta_{3},\theta_{4})=(1/10,3/10,9/10)$.
Let $(a,b)=(-1,2)$, then $\{2\theta_{1},0,2/5,2/5\}$ is not ``good'',
which implies $2/5\leq2\theta_{1}\leq3/5$, contradiction.
\item If $(r,r_{4})=(1/4,1/2)$, then $(\theta_{2},\theta_{3},\theta_{4})=(1/8,3/8,1/8)$.
Let $(a,b)=(1,2)$, then $\{2\theta_{1},1/2,0,1/4\}$ is not ``good'',
which implies $\theta_{1}=1/8$. This is the \textbf{case (\ref{case10})}
of Theorem \ref{thm1}. It is easy to check that it is indeed not
``good'' and minimal.
\item If $(r,r_{4})=(1/3,1/2)$, then $(\theta_{2},\theta_{3},\theta_{4})=(1/12,11/12,1/4)$.
Let $(a,b)=(1,4)$, then $\{4\theta_{1},1/3,0,1/2\}$ is not ``good'',
which implies $\theta_{1}=1/12$. This is the \textbf{case (\ref{case11})}
of Theorem \ref{thm1}. It is easy to check that it is indeed not
``good'' and minimal.
\item Otherwise, such $(\theta_{2},\theta_{3},\theta_{4})$ does not exist.
\end{itemize}
Assume that $(r,r_{4})\notin D_{2}$, $r\neq r_{4}$, and $\cup_{i=1}^{3}A_{i}\neq\varnothing$.
We note that $A_{1}\cup A_{2}\cup A_{4}\neq\mathbb{Z}$, $A_{2}\cup A_{5}\cup A_{7}\neq\mathbb{Z}$,
and $A_{1}\cup A_{6}\cup A_{8}\neq\mathbb{Z}$ are automatic for each
case. If $A_{3}\cup A_{4}=\mathbb{Z}$, then $(r,r_{4})=(1/4,1/2)$
and $(\theta_{2},\theta_{3},\theta_{4})\equiv(\theta_{2},1/2-\theta_{2},0)\textup{ mod }1$,
which implies $\theta_{1}\equiv\pm\theta_{2}\textup{ mod }1$ or $\theta_{1}\equiv1/2\pm\theta_{2}\textup{ mod }1$.
Without loss of generality, we can assume that $\theta_{1}=\theta_{2}$.
Let $(a,b)=(1,2)$, then $\{2\theta_{1},1/4+2\theta_{1},1/4-2\theta_{1},1/2\}$
is not ``good'', which implies $4\theta_{1}\equiv0,1/4,3/4\textup{ mod }1$.
By the minimality of $S$, $\theta_{1}=1/16$. However, when $\theta_{1}=1/16$,
$\{1/8,3/8,1/8,1/2\}$ is ``good'', so $A_{3}\cup A_{4}=\mathbb{Z}$
cannot happen. If, for example, $A_{1}\cup A_{3}\cup A_{5}\cup A_{7}=\mathbb{Z}$,
then $\theta_{2},\theta_{4}\in(1/n)\mathbb{Z}/\mathbb{Z}$, where
$n=\textup{lcm}(d_{1},d_{3})$. We can use \textbf{computer programs}
to find all $(\theta_{2},\theta_{4})$ that satisfy this condition.
The answer is:
\begin{itemize}
\item If $(r,r_{4})=(1/8,3/8)$, then $(\theta_{2},\theta_{4})=(0,1/2)$.
Let $(a,b)=(0,1)$, then $\{\theta_{1},0,\theta_{3},1/2\}$ is not
``good'', which implies $\theta_{1}\equiv\pm\theta_{3}\textup{ mod }1$.
Let $(a,b)=(4,1)$, then $\{\theta_{1},1/2,1/2+\theta_{3},0\}$ is
not ``good'', which implies $\theta_{1}\equiv1/2\pm\theta_{3}\textup{ mod }1$.
If, for example, $\theta_{1}\equiv\theta_{3}\equiv1/2-\theta_{3}\textup{ mod }1$,
then $\theta_{1}=1/4$, contradiction.
\item If $(r,r_{4})=(1/4,1/2)$, then $(\theta_{2},\theta_{4})=(0,1/2)$.
Let $(a,b)=(2,1)$, then $\{\theta_{1},1/2,1/2+\theta_{3},1/2\}$
is not ``good'', which implies $\theta_{1}=1/2$ or $\theta_{3}=0$,
contradiction.
\item Otherwise, such $(\theta_{2},\theta_{4})$ does not exist.
\end{itemize}
As the case $(r,r_{4})\in D_{1}\cap D_{2}$ before, we can use \textbf{computer
programs} to check that for any $(r,r_{4})\notin D_{1}\cap D_{2}$
with $r\neq r_{4}$, any $\theta_{1},\theta_{2},\theta_{3},\theta_{4}\in(1/n)\mathbb{Z}/\mathbb{Z}$,
and any $a,b\in\mathbb{Z}$ such that $\gcd(a,b,n)=1$, where $n=\textup{lcm}(d_{1},d_{3})$,
whether $\{b\theta_{1},ar+b\theta_{2},ar+b\theta_{3},ar_{4}+b\theta_{4}\}$
is ``good'' or not. The answer is: if $S$ is not ``good'', then
$S$ is the \textbf{case (\ref{case6})} or \textbf{case (\ref{case7})}
of Theorem \ref{thm1}.

Assume that $r=r_{4}=1/4$ or $1/3$. Without loss of generality,
we also assume that $A_{1}\cup A_{3}\cup A_{5}\neq\mathbb{Z}$, $A_{2}\cup A_{3}\cup A_{6}\neq\mathbb{Z}$,
$A_{4}\cup A_{5}\neq\mathbb{Z}$, and $A_{4}\cup A_{6}\neq\mathbb{Z}$.
If $\cup_{i=1}^{3}A_{i}=\varnothing$, we must have $r=r_{4}=1/3$
and $\{\theta_{2}+\theta_{3},\theta_{2}+\theta_{4},\theta_{3}+\theta_{4}\}\equiv\{0,1/3,2/3\}\textup{ mod }1$,
which implies $(\theta_{2},\theta_{3},\theta_{4})=(1/6,1/2,5/6)$.
Let $(a,b)=(1,2)$, then $\{2\theta_{1},1/3,1/3,0\}$ is not ``good'',
which implies $\theta_{1}=1/6$. This is the \textbf{case (\ref{case9})}
of Theorem \ref{thm1}. It is easy to check that it is indeed not
``good'' and minimal. We note that $A_{3}\cup A_{4}\neq\mathbb{Z}$,
$A_{2}\cup A_{5}\neq\mathbb{Z}$, and $A_{1}\cup A_{6}\neq\mathbb{Z}$
are automatic for each case, so if $\cup_{i=1}^{3}A_{i}\neq\varnothing$,
then $\theta_{1},\theta_{2},\theta_{3},\theta_{4}\in(1/d_{1})\mathbb{Z}/\mathbb{Z}$,
which implies $S$ is the \textbf{case (\ref{case3})} or \textbf{case
(\ref{case4})} of Theorem \ref{thm1}. It is easy to check that both
of them are indeed not ``good'' and minimal.

Assume that $r=r_{4}=1/2$ and $0\leq\theta_{2}<\theta_{3}<\theta_{4}\leq1/2$.
Let $(a,b)=(0,1)$, then $\{\theta_{1},\theta_{2},\theta_{3},\theta_{4}\}$
is not ``good'', which implies $\theta_{1}=\theta_{3}$. Let $(a,b)=(1,1)$,
then $\{\theta_{1},1/2+\theta_{2},1/2+\theta_{1},1/2+\theta_{4}\}$
is not ``good'', which implies $\theta_{1}=1/4$. By the minimality
of $S$, $\textup{ord}((1/2,\theta_{2})),\textup{ord}((1/2,\theta_{4}))\leq4$,
so $\theta_{2}=0$ and $\theta_{4}=1/2$. This is the \textbf{case
(\ref{case5})} of Theorem \ref{thm1}. It is easy to check that it
is indeed not ``good'' and minimal.

Assume that $A_{1}\cup A_{2}\cup A_{4}=\mathbb{Z}$ and $r=1/4$.
Then $(\theta_{2},\theta_{3})=(0,1/2)$. Let $(a,b)=(2c,1)$, then
$\{\theta_{1},0,1/2,2cr_{4}+\theta_{4}\}$ is not ``good'', which
implies $\theta_{1}\equiv\pm(2cr_{4}+\theta_{4})\textup{ mod }1$
for any $c\in\mathbb{Z}$. Now $\theta_{1}\equiv2cr_{4}+\theta_{4}\equiv2c'r_{4}+\theta_{4}\textup{ mod }1$
or $\theta_{1}\equiv-2cr_{4}-\theta_{4}\equiv-2c'r_{4}-\theta_{4}\textup{ mod }1$
for some $0\leq c<c'\leq2$, which implies $r_{4}=1/4$ or $1/2$.
If $r_{4}=1/4$, then $\theta_{4}\equiv-2r_{4}-\theta_{4}\textup{ mod }1$,
which implies $\theta_{4}=1/4$ or $3/4$, and therefore $\theta_{1}=1/4$.
These two cases are clearly equivalent and have been found previously.
If $r_{4}=1/2$, then $\theta_{1}=\theta_{4}$. By the minimality
of $S$, $\theta_{1}\in\{1/(2R):R\in\mathbb{Z},R\geq2\}$. This is
the \textbf{case (\ref{case2})} of Theorem \ref{thm1}. Any element
within its $\textup{SL}_{2}(\mathbb{Z})$-orbit must be of the form
\begin{align*}
 & \{(r,\theta),(r,\theta+1/2),(0,1/4),(1/2,1/4)\},\\
 & \{(r,\theta),(r+1/2,\theta),(1/4,0),(1/4,1/2)\},\\
 & \{(r,\theta),(r+1/2,\theta+1/2),(1/4,1/4),(1/4,3/4)\},
\end{align*}
so it is indeed not ``good''. If $S$ is not minimal, then the minimal
element must be $\{(0,\theta_{1}),(0,1/4),(0,1/2-\theta_{1}),(1/2,1/4)\}$.
Therefore, there exist $a,b\in\mathbb{Z}$ such that $\{a/4,a/4+b/2\}\equiv\{0,1/2\}\textup{ mod }1$,
which implies $b$ is odd, and $\{b\theta_{1},a/2+b\theta_{1}\}\equiv\{0,0\}\textup{ mod }1$,
which implies $b$ is even, contradiction.

Assume that $A_{1}\cup A_{2}\cup A_{4}=\mathbb{Z}$ and $r=1/3$.
Then $(\theta_{2},\theta_{3})=(0,1/3)$ and $\theta_{1}\leq1/3$.
If $\theta_{1}=1/3$, then by the minimality of $S$, $(r_{4},\theta_{4})=(1/3,2/3)$,
$(1/2,0)$, or $(1/2,1/2)$. The first case has been found previously.
Let $(a,b)=(0,1)$ for the second and  $(a,b)=(-1,1)$ for the third,
contradiction. Now we assume that $\theta_{1}<1/3$. Let $(a,b)=(3c,1)$,
then $\{\theta_{1},0,1/3,3cr_{4}+\theta_{4}\}$ is not ``good'',
which implies $\theta_{1}\equiv\pm(3cr_{4}+\theta_{4})\textup{ mod }1$
for any $c\in\mathbb{Z}$. Now $\theta_{1}\equiv3cr_{4}+\theta_{4}\equiv3c'r_{4}+\theta_{4}\textup{ mod }1$
or $\theta_{1}\equiv-3cr_{4}-\theta_{4}\equiv-3c'r_{4}-\theta_{4}\textup{ mod }1$
for some $0\leq c<c'\leq2$, which implies $r_{4}=1/3$ or $1/2$.
If $r_{4}=1/2$, then $\theta_{4}\equiv-3r_{4}-\theta_{4}\textup{ mod }1$,
which implies $\theta_{4}=1/4$ and also $\theta_{1}=1/4$. Let $(a,b)=(1,2)$,
contradiction. If $r_{4}=1/3$, then $\theta_{1}\equiv\pm\theta_{4}\textup{ mod }1$.
Let $(a,b)=(-1,1)$, then $\{\theta_{1},1/3,0,\theta_{4}-1/3\}$ is
not ``good'', which implies $\theta_{1}\equiv\pm(\theta_{4}-1/3)\textup{ mod }1$.
If $\theta_{1}\equiv-\theta_{4}\equiv\theta_{4}-1/3\textup{ mod }1$,
then $(\theta_{1},\theta_{4})=(1/3,2/3)$, which has been found previously.
If $\theta_{1}\equiv\theta_{4}\equiv-\theta_{4}+1/3\textup{ mod }1$,
then $(\theta_{1},\theta_{4})=(1/6,1/6)$. This is the \textbf{case
(\ref{case8})} of Theorem \ref{thm1}. It is easy to check that it
is indeed not ``good'' and minimal.

Assume that $A_{5}\cup A_{6}=\mathbb{Z}$, and $r\neq0$ or $1/2$.
Then $r+r_{4}=1/2$ and $(\theta_{3},\theta_{4})\equiv(\theta_{2}+1/2,-\theta_{2})\textup{ mod }1$.
Let $(a,b)=(2c,1)$, then $\{\theta_{1},2cr+\theta_{2},2cr+\theta_{2}+1/2,2cr+\theta_{2}\}$
is not ``good''. If $2cr+\theta_{2}\equiv0$ or $1/2\textup{ mod }1$,
then $\theta_{1}=0$ or $1/2$, contradiction. By the minimality of
$S$, $\min\{F(\theta_{1}),F(2cr+\theta_{2}),F(2cr+\theta_{2}+1/2),F(2cr+\theta_{2})\}=F(\theta_{1})$,
which implies $0<F(2cr+\theta_{2})\leq1/4$ for any $c\in\mathbb{Z}$.
Since $0<2r\leq1/2$, we must have $F(\theta_{2})=1/4$ and also $r=r_{4}=1/4$,
then $(r,\theta_{3})=(r_{4},\theta_{4})$, contradiction.

Assume that $r=0$ and $0\leq\theta_{1}<\theta_{2}<\theta_{3}\leq1/2$.
Clearly, $r_{4}\neq0$. Let $(a,b)=(c,1)$, then $\{\theta_{1},\theta_{2},\theta_{3},cr_{4}+\theta_{4}\}$
is not ``good'', which implies $\theta_{2}\equiv\pm(cr_{4}+\theta_{4})\textup{ mod }1$
for any $c\in\mathbb{Z}$. Now $\theta_{2}\equiv cr_{4}+\theta_{4}\equiv c'r_{4}+\theta_{4}\textup{ mod }1$
or $\theta_{2}\equiv-cr_{4}-\theta_{4}\equiv-c'r_{4}-\theta_{4}\textup{ mod }1$
for some $0\leq c<c'\leq2$. Since $r_{4}\neq0$, we have $c=0$,
$c'=2$, and $r_{4}=1/2$. Since $\theta_{4}\equiv-r_{4}-\theta_{4}\textup{ mod }1$,
we have $\theta_{4}=1/4$ and also $\theta_{2}=1/4$. Let $(a,b)=(1,2)$,
then $\{2\theta_{1},1/2,2\theta_{3},0\}$ is not ``good'', which
implies $\theta_{1}+\theta_{3}=1/2$. By the minimality of $S$, $\theta_{1}\in\{0\}\cup\{1/(2R):R\in\mathbb{Z},R\geq3\}$.
This is the \textbf{case (\ref{case1})} of Theorem \ref{thm1}. As
the case $A_{1}\cup A_{2}\cup A_{4}=\mathbb{Z}$ and $r=1/4$ before,
it is indeed not ``good'' and minimal.
\end{proof}

\section{Proof of Theorem \texorpdfstring{\ref{thm1}}{1}: Step (III)}
\begin{thm}
Let $\{(r_{i},\theta_{i})\}_{i=1}^{4}\in\mathcal{R}_{4}$ be one of
the quadruples listed in Theorem \ref{thm1}, $n$ their common order,
$\{w_{1},w_{2}\}$ a basis of $E[n]$. Write $E[r_{i},\theta_{i}]=[nr_{i}]w_{1}+[n\theta_{i}]w_{2}$.
Then whenever $\{E[r_{i},\theta_{i}]\}_{i=1}^{4}$ is well-defined,
i.e.,
\begin{itemize}
\item $\textup{char}(K)\neq2$ for the cases (\ref{case1}), (\ref{case2}),
and (\ref{case5}),
\item $\textup{char}(K)\nmid n$ for the cases (\ref{case3}), (\ref{case4}),
(\ref{case6}), (\ref{case7}), (\ref{case10}), and (\ref{case11}),
\item $\textup{char}(K)\neq3$ and there exists a point of order $2$ (i.e.,
$E$ is ordinary if $\textup{char}(K)=2$) for the cases (\ref{case8})
and (\ref{case9}),
\end{itemize}
the image of it under $\pi$ inside $\mathfrak{S}_{4}\backslash(\mathbb{P}^{1})^{4}/\textup{PGL}_{2}$
is a constant.
\end{thm}

\begin{proof}
Let us first collect some basic facts that we will need for the Jacobian
form and Hessian form.

If $\textup{char}(K)\neq2$, then any $E/K$ can be transformed to
the Jacobian form
\[
E_{\delta}:y^{2}=x^{4}-(\delta^{2}+\delta^{-2})x^{2}+1,
\]
where $\delta^{4}\neq0,1$. We take the origin to be $O_{\delta}=(\delta,0)$
and $\pi_{\delta}:E_{\delta}\rightarrow\mathbb{P}^{1},(x,y)\mapsto x$.
Its $2$-torsion points are $(\pm\delta^{\pm1},0)$, which induce
\[
\pi_{\delta}(P+(\pm\delta^{\pm1},0))=\pm\pi_{\delta}(P)^{\pm1}.
\]
The fixed points of those nontrivial ones $\{0,\infty\}$, $\{\pm1\},$
and $\{\pm i\}$ constitute $\pi_{\delta}(E_{\delta}^{*}[4])$, where
$i$ is a primitive fourth root of $K$. Moreover,
\[
\pi_{\delta}([2]P)=-\frac{\delta^{2}-2\pi_{\delta}(P)^{2}+\delta^{2}\pi_{\delta}(P)^{4}}{\delta(1-2\delta^{2}\pi_{\delta}(P)^{2}+\pi_{\delta}(P)^{4})}.
\]

If $\textup{char}(K)\neq3$, then any $E/K$ can be transformed to
the Hessian form
\[
E_{\lambda}:x^{3}+y^{3}+z^{3}=3\lambda xyz,
\]
where $\lambda^{3}\neq1$. We take the origin to be $O_{\lambda}=(1:-1:0)$
and $\pi_{\lambda}:E_{\lambda}\rightarrow\mathbb{P}^{1},(x:y:z)\mapsto-(x+y)/z$.
Its $3$-torsion points are
\[
\begin{array}{ccc}
(1:-1:0), & (1:-\omega:0), & (1:-\omega^{2}:0),\\
(0:1:-1), & (0:1:-\omega), & (0:1:-\omega^{2}),\\
(-1:0:1), & (-\omega:0:1), & (-\omega^{2}:0:1),
\end{array}
\]
where $\omega$ is a primitive cube root of $K$. The addition formula
is
\[
(x_{1}:y_{1}:z_{1})+(x_{2}:y_{2}:z_{2})=(y_{1}^{2}x_{2}z_{2}-y_{2}^{2}x_{1}z_{1}:x_{1}^{2}y_{2}z_{2}-x_{2}^{2}y_{1}z_{1}:z_{1}^{2}x_{2}y_{2}-z_{2}^{2}x_{1}y_{1}).
\]
The doubling formula is
\[
[2](x:y:z)=(y(x^{3}-z^{3}):x(z^{3}-y^{3}):z(y^{3}-x^{3})).
\]

We use the Jacobian form for the cases (\ref{case1}), (\ref{case2}),
and (\ref{case5}).\\
Let $Q\in E_{\delta}^{*}[2]$, $Q_{1},Q_{2}\in E_{\delta}^{*}[4]$
such that $[2]Q_{1}=[2]Q_{2}=Q$ and $Q_{1}+Q_{2}\neq O_{\delta}$,
and $P\in E_{\delta}[\infty]$ such that $P\neq Q_{1},Q_{2}$. Then
$\pi_{\delta}$ maps $(P,P+Q,Q_{1},Q_{2})$ to $(\pi_{\delta}(P),-\pi_{\delta}(P),0,\infty)$,
whose cross ratio is a constant.

We use the Jacobian form for the cases (\ref{case4}) and (\ref{case10}).\\
The case (\ref{case4}) is done. Now we deal with the case (\ref{case10}).
Fix a basis $\{w_{1},w_{2}\}$ of $E_{\delta}[8]$ such that $[4]w_{2}=(-\delta,0)$
and $[4]w_{1}=(\delta^{-1},0)$. Let $a=\pi_{\delta}(w_{2})$ and
$b=\pi_{\delta}([2]w_{1}+w_{2})$, then $\pi_{\delta}([2]w_{1}+[3]w_{2})=-b^{-1}$
and $\pi_{\delta}([4]w_{1}+w_{2})=a^{-1}$. The cross ratio of $(a,b,-b^{-1},a^{-1})$
is
\[
\frac{(a-b)(a^{-1}+b^{-1})}{(a+b^{-1})(a^{-1}-b)}=\frac{a^{2}-b^{2}}{1-a^{2}b^{2}}.
\]
Let $\pi_{\delta}([2]w_{2})=0$ and $\pi_{\delta}([4]w_{1}+[2]w_{2})=\infty$,
then we know that
\begin{eqnarray*}
\delta^{2}-2a^{2}+\delta^{2}a^{4}=0 & \Rightarrow & a^{2}=\delta^{-2}\pm\sqrt{\delta^{-4}-1}=\delta^{-2}\pm i\delta^{-2}\sqrt{\delta^{4}-1},\\
1-2\delta^{2}b^{2}+b^{4}=0 & \Rightarrow & b^{2}=\delta^{2}\pm\sqrt{\delta^{4}-1}.
\end{eqnarray*}
If, for example, taking ``$+$'' for both, then the cross ratio
\begin{eqnarray*}
\frac{a^{2}-b^{2}}{1-a^{2}b^{2}} & = & \frac{(\delta^{-2}+i\delta^{-2}\sqrt{\delta^{4}-1})-(\delta^{2}+\sqrt{\delta^{4}-1})}{1-(\delta^{-2}+i\delta^{-2}\sqrt{\delta^{4}-1})(\delta^{2}+\sqrt{\delta^{4}-1})}\\
 & = & \frac{\delta^{-2}+i\delta^{-2}\sqrt{s^{4}-1}-\delta^{2}-\sqrt{\delta^{4}-1}}{1-1-\delta^{-2}\sqrt{\delta^{4}-1}-i\sqrt{\delta^{4}-1}-i\delta^{-2}(\delta^{4}-1)}\\
 & = & \frac{(\delta^{-2}-\delta^{2})+(i\delta^{-2}-1)\sqrt{\delta^{4}-1}}{(i\delta^{-2}-i\delta^{2})+(-\delta^{-2}-i)\sqrt{\delta^{4}-1}}\\
 & = & -i
\end{eqnarray*}
is a constant.

We use the Hessian form for the cases (\ref{case3}), (\ref{case8}),
and (\ref{case9}).\\
The case (\ref{case3}) is done. Now we deal with the cases (\ref{case8})
and (\ref{case9}). Let $Q=(a:a:1)\in E_{\lambda}^{*}[2]$, $P_{0}=(1:-\omega:0)$,
$P_{1}=(-1:0:1)$, $P_{2}=(-\omega:0:1)$, and $P_{3}=(-\omega^{2}:0:1)$.
Then $\pi_{\lambda}$ maps $(P_{0},P_{1},P_{0}+Q,P_{1}+Q,P_{2}+Q,P_{3}+Q)$
to $(\infty,1,a,-1-a^{-1},-\omega-\omega^{2}a^{-1},-\omega^{2}-\omega a^{-1})$.
The cross ratio of $(\infty,1,-\omega-\omega^{2}a^{-1},-\omega^{2}-\omega a^{-1})$
is
\[
\frac{1+\omega+\omega^{2}a^{-1}}{1+\omega^{2}+\omega a^{-1}}=\frac{-\omega^{2}+\omega^{2}a^{-1}}{-\omega+\omega a^{-1}}=\omega.
\]
The cross ratio of $(a,-1-a^{-1},-\omega-\omega^{2}a^{-1},-\omega^{2}-\omega a^{-1})$
is
\begin{eqnarray*}
 &  & \frac{(a+\omega+\omega^{2}a^{-1})(-1-a^{-1}+\omega^{2}+\omega a^{-1})}{(a+\omega^{2}+\omega a^{-1})(-1-a^{-1}+\omega+\omega^{2}a^{-1})}\\
 & = & \frac{(a+\omega+\omega^{2}a^{-1})(a^{-1}-\omega^{2})(\omega-1)}{(a+\omega^{2}+\omega a^{-1})(a^{-1}-\omega)(\omega^{2}-1)}\\
 & = & \frac{\omega^{2}(a^{-2}-a)(\omega-1)}{\omega(a^{-2}-a)(\omega^{2}-1)}\\
 & = & -\omega^{2}.
\end{eqnarray*}
Both of them are constants.

We use the Jacobian form for the cases (\ref{case6}) and (\ref{case7}).\\
Let $P_{0},P_{1},P_{2},P_{3}\in E_{\delta}^{*}[3]$ such that $P_{i}+P_{j}\neq O_{\delta}$,
and $a,b,c,d$ their images under $\pi_{\delta}$. We have already
known that the cross ratio of $(a,b,-c,-d)$ is a constant, so the
cross ratio of $(a^{-1},b^{-1},-c^{-1},-d^{-1})$ is the same constant.
The cross ratio of $(a,-b,c^{-1},-d^{-1})$ is
\[
\frac{(a+b)(c^{-1}+d^{-1})}{(a+d^{-1})(c^{-1}+b)}=\frac{(ad+bc)+(ac+bd)}{(ad+bc)+(abcd+1)}.
\]
We know that (see, for example, \cite[Proof of Corollary 3.6 (C)]{MR3536148})
\[
(x-a)(x-b)(x-c)(x-d)=x^{4}+2\delta x^{3}-2\delta^{-1}x-1,
\]
which implies
\[
(x-(ab+cd))(x-(ac+bd))(x-(ad+bc))=x^{3}+4(\delta^{2}+\delta^{-2}),
\]
so $abcd=-1$ and $\dfrac{ac+bd}{ad+bc}=\omega$. Therefore,
\[
\frac{(ad+bc)+(ac+bd)}{(ad+bc)+(abcd+1)}=-\omega^{2}
\]
is a constant.

We use the Hessian form for the case (\ref{case11}).\\
Let $P_{1}=(-1:0:1)$, $P_{2}=(-\omega:0:1)$, $P_{3}=(-\omega^{2}:0:1)$,
$\{v_{1},v_{2}\}$ a basis of $E_{\lambda}[4]$, $v_{2}=(a_{1}:b_{1}:1)$,
$[2]v_{1}+v_{2}=(a_{2}:b_{2}:1)$, and $[2]v_{2}=(c:c:1)$. Then $\pi_{\lambda}$
maps $(P_{1}+v_{2},P_{2}+v_{2},P_{3}+v_{2},[2]v_{1}+v_{2})$ to $(-a_{1}^{-1}(b_{1}+1),-a_{1}^{-1}(\omega b_{1}+\omega^{2}),-a_{1}^{-1}(\omega^{2}b_{1}+\omega),-(a_{2}+b_{2}))$,
whose cross ratio is
\begin{eqnarray*}
 &  & \frac{(a_{1}^{-1}(b_{1}+1)-a_{1}^{-1}(\omega b_{1}+\omega^{2}))(a_{2}+b_{2}-a_{1}^{-1}(\omega^{2}b_{1}+\omega))}{(a_{1}^{-1}(b_{1}+1)-a_{1}^{-1}(\omega^{2}b_{1}+\omega))(a_{2}+b_{2}-a_{1}^{-1}(\omega b_{1}+\omega^{2}))}\\
 & = & \frac{(1-\omega)(b_{1}-\omega^{2})(a_{1}(a_{2}+b_{2})-(\omega^{2}b_{1}+\omega))}{(1-\omega^{2})(b_{1}-\omega)(a_{1}(a_{2}+b_{2})-(\omega b_{1}+\omega^{2}))}\\
 & = & \frac{(1-\omega)(a_{1}(b_{1}-\omega^{2})(a_{2}+b_{2})-\omega^{2}b_{1}^{2}+1)}{(1-\omega^{2})(a_{1}(b_{1}-\omega)(a_{2}+b_{2})-\omega b_{1}^{2}+1)}\\
 & = & \frac{(1-\omega)(-\omega^{2}(a_{1}a_{2}+a_{1}b_{2}-b_{1}^{2})+a_{1}b_{1}(a_{2}+b_{2})+1)}{(1-\omega^{2})(-\omega(a_{1}a_{2}+a_{1}b_{2}-b_{1}^{2})+a_{1}b_{1}(a_{2}+b_{2})+1)}\\
 & = & -\omega^{2},
\end{eqnarray*}
if we have $a_{1}b_{1}(a_{2}+b_{2})=-1$. Now let us check this is
true. From the doubling formula, we know that $(a_{1},b_{1})$, $(b_{1},a_{1})$,
$(a_{2},b_{2})$, and $(b_{2},a_{2})$ are the solutions of $\left(\dfrac{y(x^{3}-1)}{y^{3}-x^{3}},\dfrac{x(1-y^{3})}{y^{3}-x^{3}}\right)=(c,c)$.
From the second coordinate, we get $y^{3}=\dfrac{cx^{3}+x}{x+c}$.
Substituting it into the first coordinate, we get $y=-\dfrac{cx}{x+c}$,
i.e., $x^{-1}+y^{-1}=-c^{-1}$. Thus we have $\dfrac{cx^{3}+x}{x+c}=y^{3}=\left(-\dfrac{cx}{x+c}\right)^{3}$,
which can be simplified to $cx^{4}+2c^{2}x^{3}+(2c^{3}+1)x^{2}+2cx+c^{2}=0,$
so $a_{1}b_{1}a_{2}b_{2}=c$. Therefore, $a_{1}b_{1}(a_{2}+b_{2})=a_{1}b_{1}a_{2}b_{2}(a_{2}^{-1}+b_{2}^{-1})=-1$.
\end{proof}

\section{A Corollary of Theorem \texorpdfstring{\ref{thm1}}{1}}
\begin{thm}
\textup{\label{thm8}\cite[Page 412, Proposition 1.3]{MR1312368}}
Define a normalized theta function $\Theta(u,q)$ by the formula
\[
\Theta(u,q)=(1-u)\prod_{n\geq1}\frac{(1-q^{n}u)(1-q^{n}u^{-1})}{(1-q^{n})^{2}}.
\]
\begin{itemize}
\item $\Theta(u,q)$ converges for all $u,q\in\mathbb{C}^{*}$ with $\left|q\right|<1$
and satisfies the functional equation
\[
\Theta(qu,q)=\Theta(u^{-1},q)=-u^{-1}\Theta(u,q).
\]
\item $\Theta(u,q)$ is related to the function $X(u,q)$ described in Theorem
\ref{thm2} by the formula
\[
X(u_{1},q)-X(u_{2},q)=-\frac{u_{2}\Theta(u_{1}u_{2},q)\Theta(u_{1}u_{2}^{-1},q)}{\Theta(u_{1},q)^{2}\Theta(u_{2},q)^{2}}.
\]
\end{itemize}
\end{thm}

\begin{cor}
Suppose that $\{(r_{i},\theta_{i})\}_{i=1}^{4}\in\mathcal{R}_{4}$
gives a constant image. Write $u_{i}=\exp(2\pi\sqrt{-1}\theta_{i})q^{r_{i}}$,
then for $\left|q\right|<1$, we have the infinite product identity
\[
\frac{\Theta(u_{1}u_{2},q)\Theta(u_{1}u_{2}^{-1},q)\Theta(u_{3}u_{4},q)\Theta(u_{3}u_{4}^{-1},q)}{\Theta(u_{1}u_{4},q)\Theta(u_{1}u_{4}^{-1},q)\Theta(u_{3}u_{2},q)\Theta(u_{3}u_{2}^{-1},q)}=\textup{constant}.
\]
\end{cor}

\begin{proof}
By Theorem \ref{thm8}, the cross ratio of $\{X(u_{i},q)\}_{i=1}^{4}$
can be written as
\[
\frac{(X(u_{1},q)-X(u_{2},q))(X(u_{3},q)-X(u_{4},q))}{(X(u_{1},q)-X(u_{4},q))(X(u_{3},q)-X(u_{2},q))}=\frac{\Theta(u_{1}u_{2},q)\Theta(u_{1}u_{2}^{-1},q)\Theta(u_{3}u_{4},q)\Theta(u_{3}u_{4}^{-1},q)}{\Theta(u_{1}u_{4},q)\Theta(u_{1}u_{4}^{-1},q)\Theta(u_{3}u_{2},q)\Theta(u_{3}u_{2}^{-1},q)}.
\]
\end{proof}

\section{Further Discussion}

Let $S=\{(r_{i},\theta_{i})\}_{i=1}^{4}\in\mathcal{R}_{4}$ and $n$
their common order. Let
\[
\mathcal{H}^{*}=\{\tau\in\mathbb{C}:\textup{Im}(\tau)>0\}\cup\mathbb{P}^{1}(\mathbb{Q})
\]
 be the extended upper half plane, 
\[
\Gamma(n)=\left\{ \gamma\in\textup{SL}_{2}(\mathbb{Z}):\gamma\equiv\begin{pmatrix}1 & 0\\
0 & 1
\end{pmatrix}\textup{ mod }n\right\} 
\]
the principal congruence subgroup of level $n$, and $X(n)=\Gamma(n)\backslash\mathcal{H}^{*}$.
Let
\[
q=\exp(2\pi\sqrt{-1}\tau),u_{i}=\exp(2\pi\sqrt{-1}\theta_{i})q^{r_{i}},
\]
and $\mu_{S}$ the composition of
\begin{eqnarray*}
X(n) & \to & (\mathbb{P}^{1})^{4}\\
\tau & \mapsto & (X(u_{1},q),X(u_{2},q),X(u_{3},q),X(u_{4},q))
\end{eqnarray*}
and
\begin{eqnarray*}
(\mathbb{P}^{1})^{4}\;\xrightarrow{\textup{cross ratio}}\;(\mathbb{Z}/2\mathbb{Z})^{2}\backslash(\mathbb{P}^{1})^{4}/\textup{PGL}_{2} & \xrightarrow{\mathfrak{S}_{3}} & \mathfrak{S}_{4}\backslash(\mathbb{P}^{1})^{4}/\textup{PGL}_{2}\\
z & \mapsto & \frac{(z^{2}-z+1)^{3}}{z^{2}(z-1)^{2}}.
\end{eqnarray*}
Then $\mu_{S}$ is a meromorphic function from $X(n)$ to $\mathbb{P}^{1}$.
Let
\[
\Gamma_{S}=\{\gamma^{-T}\in\textup{SL}_{2}(\mathbb{Z}):\gamma\cdot S=S\}
\]
and
\[
\Delta_{S}=\{\gamma\in\textup{SL}_{2}(\mathbb{Z}):\mu_{S}(\gamma\cdot\tau)=\mu_{S}(\tau)\textup{ for any }\tau\in X(n)\}.
\]
By the proof of Corollary \ref{cor5}, $\Gamma(n)\subseteq\Gamma_{S}\subseteq\Delta_{S}\subseteq\textup{SL}_{2}(\mathbb{Z})$.
The map $\mu_{S}$ factors through $X(\Gamma_{S})=\Gamma_{S}\backslash\mathcal{H}^{*}$
and $X(\Delta_{S})=\Delta_{S}\backslash\mathcal{H}^{*}$. In summary,
we have the following commutative diagram:
\[
\xymatrix{X(n)\ar[d]\ar@{=}[r] & \Gamma(n)\backslash\mathcal{H}^{*}\ar[d]\ar[r]\ar@/^{0.25in}/[rrrr]\sp-{\mu_{S}} & (\mathbb{P}^{1})^{4}\ar[rr]\sp-{\textup{cross ratio}} &  & (\mathbb{Z}/2\mathbb{Z})^{2}\backslash(\mathbb{P}^{1})^{4}/\textup{PGL}_{2}\ar[r]\sp-{\mathfrak{S}_{3}} & \mathfrak{S}_{4}\backslash(\mathbb{P}^{1})^{4}/\textup{PGL}_{2}\\
X(\Gamma_{S})\ar[d]\ar@{=}[r] & \Gamma_{S}\backslash\mathcal{H}^{*}\ar[d]\ar[urrrr]\\
X(\Delta_{S})\ar[d]\ar@{=}[r] & \Delta_{S}\backslash\mathcal{H}^{*}\ar[d]\ar[uurrrr]\\
X(1)\ar@{=}[r] & \textup{SL}_{2}(\mathbb{Z})\backslash\mathcal{H}^{*}
}
\]
Theorem \ref{thm1} classifies all $S$ such that $\mu_{S}$ is a
constant map, thus gives a complete answer to \cite[Conjecture 3]{Hitchin}.
The next natural question is whether
\begin{conjecture}
\textup{\cite[Conjecture 4]{Hitchin}} $\Gamma_{S}=\Delta_{S}$ for
all but finitely many $S$, up to the $\textup{SL}_{2}(\mathbb{Z})$-equivalence
in $\mathcal{R}_{4}$.
\end{conjecture}

Now we give an example to show that sometimes $\Gamma_{S}\neq\Delta_{S}$
can indeed happen. Consider
\begin{eqnarray*}
S & = & \{(0,1/5),(0,2/5),(1/5,0),(2/5,0)\},\\
\begin{pmatrix}1 & 2\\
1 & 3
\end{pmatrix}\cdot S & = & \{(1/5,1/5),(1/5,4/5),(2/5,2/5),(2/5,3/5)\},\\
\begin{pmatrix}1 & 1\\
2 & 3
\end{pmatrix}\cdot S & = & \{(1/5,2/5),(1/5,3/5),(2/5,1/5),(2/5,4/5)\}.
\end{eqnarray*}
By the second part of Theorem \ref{thm8},
\begin{eqnarray*}
 &  & \frac{(X(u_{1},q)-X(u_{2},q))(X(u_{3},q)-X(u_{4},q))}{(X(u_{1},q)-X(u_{4},q))(X(u_{3},q)-X(u_{2},q))}\\
\\
 & = & \frac{\Theta(u_{1}u_{2},q)\Theta(u_{1}u_{2}^{-1},q)\Theta(u_{3}u_{4},q)\Theta(u_{3}u_{4}^{-1},q)}{\Theta(u_{1}u_{4},q)\Theta(u_{1}u_{4}^{-1},q)\Theta(u_{3}u_{2},q)\Theta(u_{3}u_{2}^{-1},q)}\\
\\
 & = & \begin{cases}
\dfrac{\Theta(\zeta_{5}^{3},q)\Theta(\zeta_{5}^{4},q)\Theta(q^{3/5},q)\Theta(q^{-1/5},q)}{\Theta(\zeta_{5}q^{2/5},q)\Theta(\zeta_{5}q^{-2/5},q)\Theta(\zeta_{5}^{2}q^{1/5},q)\Theta(\zeta_{5}^{3}q^{1/5},q)} & \textup{if }(u_{1},u_{2},u_{3},u_{4})=(\zeta_{5},\zeta_{5}^{2},q^{1/5},q^{2/5}),\\
\\
\dfrac{\Theta(q^{2/5},q)\Theta(\zeta_{5}^{2},q)\Theta(q^{4/5},q)\Theta(\zeta_{5}^{4},q)}{\Theta(\zeta_{5}^{4}q^{3/5},q)\Theta(\zeta_{5}^{3}q^{-1/5},q)\Theta(\zeta_{5}q^{3/5},q)\Theta(\zeta_{5}^{3}q^{1/5},q)} & \textup{if }(u_{1},u_{2},u_{3},u_{4})=(\zeta_{5}q^{1/5},\zeta_{5}^{4}q^{1/5},\zeta_{5}^{2}q^{2/5},\zeta_{5}^{3}q^{2/5}),\\
\\
\dfrac{\Theta(q^{2/5},q)\Theta(\zeta_{5}^{4},q)\Theta(q^{4/5},q)\Theta(\zeta_{5}^{2},q)}{\Theta(\zeta_{5}q^{3/5},q)\Theta(\zeta_{5}^{3}q^{-1/5},q)\Theta(\zeta_{5}^{4}q^{3/5},q)\Theta(\zeta_{5}^{3}q^{1/5},q)} & \textup{if }(u_{1},u_{2},u_{3},u_{4})=(\zeta_{5}^{2}q^{1/5},\zeta_{5}^{3}q^{1/5},\zeta_{5}q^{2/5},\zeta_{5}^{4}q^{2/5}),
\end{cases}
\end{eqnarray*}
where $\zeta_{5}=\exp(2\pi\sqrt{-1}/5)$. By the first part of Theorem
\ref{thm8}, these three expressions are equal for any $\left|q\right|<1$.
Actually, for this example, we have $\Gamma_{S}/\pm\Gamma(5)\cong(\mathbb{Z}/2\mathbb{Z})^{2}$
and $\Delta_{S}/\pm\Gamma(5)\cong A_{4}$. Moreover, we note that
$S$, $\begin{pmatrix}1 & 2\\
1 & 3
\end{pmatrix}\cdot S$, and $\begin{pmatrix}1 & 1\\
2 & 3
\end{pmatrix}\cdot S$ give a partition of the collection of all projective torsion points
of order $5$.

If $n=p\geq5$ is a prime, then by Lemma \ref{lem6}, $S$ is ``good''.
The calculations in the proof of Lemma \ref{lem4} imply that at some
cusp of $X(\Delta_{S})$, the $q$-expansion is not inside $\mathbb{C}((q))$.
Therefore, $\Delta_{S}\neq\textup{SL}_{2}(\mathbb{Z})$ and
\[
\Delta_{S}/\pm\Gamma(p)\subsetneq\textup{SL}_{2}(\mathbb{Z})/\pm\Gamma(p)\cong\textup{PSL}_{2}(\mathbb{Z}/p\mathbb{Z}).
\]
By \cite[Page 412, Theorem 6.25]{MR0648772}, any subgroup $H\subsetneq\textup{PSL}_{2}(\mathbb{Z}/p\mathbb{Z})$
contains at most one copy of $\mathbb{Z}/p\mathbb{Z}$. If $\mathbb{Z}/p\mathbb{Z}\subseteq H$,
then the normalizer $N_{\textup{PSL}_{2}(\mathbb{Z}/p\mathbb{Z})}(\mathbb{Z}/p\mathbb{Z})\supseteq H$
is a subgroup of order $p(p-1)/2$.
\begin{itemize}
\item If $p\nmid[\Delta_{S}:\pm\Gamma(p)]$, then the quotient map $X(p)\to X(\Delta_{S})$
is unramified at all cusps of $X(\Delta_{S})$.
\item If $p\mid[\Delta_{S}:\pm\Gamma(p)]$, then for some $\tau\in\mathbb{P}^{1}(\mathbb{Q})$,
the isotropy subgroup
\[
\Delta_{S,\tau}=\{\gamma\in\Delta_{S}:\gamma\cdot\tau=\tau\}
\]
satisfies
\[
\mathbb{Z}/p\mathbb{Z}\cong\Delta_{S,\tau}/(\Delta_{S,\tau}\cap\pm\Gamma(p))\subseteq\Delta_{S}/\pm\Gamma(p)\subsetneq\textup{SL}_{2}(\mathbb{Z})/\pm\Gamma(p)\cong\textup{PSL}_{2}(\mathbb{Z}/p\mathbb{Z}).
\]
The ramified cusps of $X(\Delta_{S})$ are corresponding to the left
cosets in
\[
N_{\textup{SL}_{2}(\mathbb{Z})/\pm\Gamma(p)}(\Delta_{S,\tau}/(\Delta_{S,\tau}\cap\pm\Gamma(p)))/(\Delta_{S}/\pm\Gamma(p)),
\]
whose size is
\[
\frac{p(p-1)/2}{[\Delta_{S}:\pm\Gamma(p)]}=\frac{[\textup{SL}_{2}(\mathbb{Z}):\Delta_{S}]}{p+1}.
\]
\end{itemize}
\textbf{Acknowledgments.} The first author was partially supported
by the HSE University Basic Research Program, Russian Academic Excellence
Project `5-100', and EPSRC programme grant EP/M024830. The second
author would like to express his gratitude for a pleasant stay at
Laboratory of Algebraic Geometry, HSE, where a substantial part of
this article was accomplished.

\addcontentsline{toc}{section}{References}

Fedor Bogomolov\\
Courant Institute of Mathematical Sciences, New York University\\
251 Mercer Street, New York, NY 10012, USA\\
Email: \href{mailto:bogomolo@cims.nyu.edu}{bogomolo@cims.nyu.edu}

\medskip{}

Fedor Bogomolov\\
Laboratory of Algebraic Geometry and its Applications\\
National Research University Higher School of Economics\\
6 Usacheva Street, 119048 Moscow, Russia

\medskip{}

Hang Fu\\
National Center for Theoretical Sciences, National Taiwan University,
Taipei, Taiwan\\
Email: \href{mailto:fu@ncts.ntu.edu.tw}{fu@ncts.ntu.edu.tw}
\end{document}